\documentclass[11pt]{amsart}
\usepackage[margin=1.1in]{geometry}
\usepackage{amsmath,amssymb}
\usepackage{graphicx}
\usepackage[dvipsnames]{xcolor}
\usepackage{enumitem}
\usepackage{booktabs}
\usepackage{array}
\usepackage[colorlinks=true,linkcolor=blue!50!black,citecolor=green!40!black,urlcolor=blue!50!black]{hyperref}

\newtheorem{theorem}{Theorem}[section]
\newtheorem{proposition}[theorem]{Proposition}
\newtheorem{lemma}[theorem]{Lemma}
\newtheorem{corollary}[theorem]{Corollary}
\newtheorem{observation}[theorem]{Observation}
\theoremstyle{definition}
\newtheorem{definition}[theorem]{Definition}
\newtheorem{problem}[theorem]{Problem}
\theoremstyle{remark}

\hypersetup{pdftitle={Diameter-two graphs whose mutual-visibility chromatic number equals their chromatic number},pdfauthor={Siddhesh Chaubal}}

\newcommand{\cmu}{\chi_\mu}
\newcommand{\R}{\text{red}}
\newcommand{\B}{\text{blue}}
\newcommand{\claimhead}[2]{\par\medskip\noindent\textbf{Claim~#1.} \textit{#2}\par\smallskip\noindent\ignorespaces}
\newcommand{\claimqed}{\hfill$\diamond$\par}

\title[Diameter-two graphs with $\chi_\mu=\chi$]{Diameter-two graphs whose mutual-visibility chromatic number equals their chromatic number}
\author{Siddhesh Chaubal}
\address{Afinit, India}
\email{spchaubal@gmail.com}


\begin{document}

\begin{abstract}
A set $S$ of vertices of a graph $G$ is a \emph{mutual-visibility set} if every pair of vertices of $S$ is joined by a shortest path whose internal vertices avoid $S$, and the \emph{mutual-visibility chromatic number} $\cmu(G)$ is the least number of mutual-visibility sets into which $V(G)$ can be partitioned. Klav\v{z}ar, Kuziak, Valenzuela Tripodoro and Yero (2025) proved that $\cmu(G)\le\chi(G)$ whenever $G$ has diameter two, observed that equality holds for complete bipartite graphs, and asked whether there are any other diameter-two graphs attaining equality. We answer this question in the affirmative. We construct a graph $G_{11}$ of order $11$ and diameter two with $\cmu(G_{11})=\chi(G_{11})=3$, and an infinite family of such graphs obtained from $G_{11}$ by replacing one vertex with pairwise non-adjacent copies of it.
\end{abstract}

\maketitle

\section{Introduction}

Mutual visibility in graphs was introduced by Di Stefano~\cite{DiStefano2022}, motivated by settings in which a set of entities placed at vertices of a network, such as mobile robots, or nodes whose messages must not pass through other members of the set, must be able to ``see'' one another along shortest paths that avoid the other members. The concept has since been studied intensively, including on graphs of diameter two~\cite{CDKY2024}. Klav\v{z}ar, Kuziak, Valenzuela Tripodoro and Yero~\cite{KKVY2025} introduced the corresponding colouring problem: partition the vertex set into as few mutual-visibility sets as possible.

\begin{definition}[\cite{DiStefano2022,KKVY2025}]
Let $G$ be a graph and $S\subseteq V(G)$. Two vertices $u,v\in S$ are called \emph{$S$-visible} when some shortest $u,v$-path avoids $S$ except at its two ends, and $S$ is called a \emph{mutual-visibility set} (MV set) if every pair of its vertices is $S$-visible. A \emph{mutual-visibility colouring} of $G$ is a partition of $V(G)$ into MV sets, and the \emph{mutual-visibility chromatic number} $\cmu(G)$ is the least number of sets in such a partition.
\end{definition}

For graphs of diameter two, every independent set is an MV set (two non-adjacent vertices of an independent set $S$ have a common neighbour, which cannot lie in $S$), and so every proper colouring is a mutual-visibility colouring:

\begin{proposition}[{\cite[Proposition~2.2]{KKVY2025}}]\label{prop:upper}
If $G$ has diameter two, then $\cmu(G)\le\chi(G)$. Equality holds for every complete bipartite graph $K_{r,s}$ with $(r,s)\neq(1,1)$.
\end{proposition}

The authors of~\cite{KKVY2025} posed the following problem (quoted from the published version).

\begin{problem}[{\cite[Problem~8.3]{KKVY2025}}]\label{prob:83}
Characterize the graphs $G$ of diameter two for which $\cmu(G)=\chi(G)$ holds. In particular, are there any additional such graphs besides the complete bipartite graphs? In addition, are there some other relationships between $\cmu(G)$ and $\chi(G)$ for other graph classes?
\end{problem}

Bre\v{s}ar, Peterin, Samadi and Yero~\cite{BPSY2026} obtained partial results towards Problem~\ref{prob:83}. Using the $1$-defective chromatic number $\chi_1(G)$ (the least $k$ such that $V(G)$ can be partitioned into $k$ sets each inducing a subgraph of maximum degree at most $1$; see~\cite{CCW1986}), they proved:

\begin{proposition}[{\cite[Proposition~18]{BPSY2026}}]\label{prop:bpsy}
If $G$ has diameter two, then $\cmu(G)\le\chi_1(G)$.
\end{proposition}

Since every proper colouring is $1$-defective, $\chi_1(G)\le\chi(G)$ for every graph $G$, so Proposition~\ref{prop:bpsy} sharpens Proposition~\ref{prop:upper}. Consequently, any diameter-two graph with $\cmu(G)=\chi(G)$ must also satisfy $\chi_1(G)=\chi(G)$. On the other hand, the $5$-cycle and the Petersen graph, for instance, satisfy $\cmu<\chi$~\cite{BPSY2026}. Despite these results, the question of whether there exist diameter-two graphs other than complete bipartite graphs satisfying $\cmu(G)=\chi(G)$ remained unresolved. We settle this existence question in the affirmative.

\subsection*{Our results}
In Section~\ref{sec:g11} we exhibit a graph $G_{11}$ of order $11$ and diameter two with $\cmu(G_{11})=\chi(G_{11})=3$, and more generally show that the graphs $G_{11}(t)$, obtained from $G_{11}$ by replacing one vertex by $t$ pairwise non-adjacent copies, satisfy $\cmu=\chi=3$ for every $t\ge1$ (Theorems~\ref{thm:g11} and~\ref{thm:family}, Corollary~\ref{cor:g11}). The proof is structural: four short claims, each a direct consequence of the neighbourhood structure of these graphs, rule out every mutual-visibility colouring with two colours.

The paper is organized as follows. Section~\ref{sec:prelim} reformulates mutual visibility for graphs of diameter two, Section~\ref{sec:g11} contains the constructions and the proofs, and Section~\ref{sec:concl} lists some open questions.

\section{A reformulation for diameter two}\label{sec:prelim}

All graphs are finite, simple and connected unless stated otherwise; $N(v)$ denotes the open neighbourhood of a vertex $v$ and $d(u,w)$ the distance between $u$ and $w$. Let $G$ have diameter two. If $u,w$ are adjacent, the edge $uw$ is a shortest $u,w$-path without internal vertices. If $u,w$ are non-adjacent, then $d(u,w)=2$, and the shortest $u,w$-paths are exactly the paths $u\,x\,w$ with $x$ a common neighbour of $u$ and $w$. Hence a set $S\subseteq V(G)$ is an MV set if and only if every two non-adjacent vertices of $S$ have a common neighbour outside $S$. Consequently, if $G$ is moreover not complete (so $\cmu(G)\ge2$), then $\cmu(G)\le 2$ if and only if there is a colouring $\varphi\colon V(G)\to\{\R,\B\}$ with the following property:
\begin{quote}
\textbf{(V)} every two non-adjacent vertices of the same colour have a common neighbour of the other colour.
\end{quote}
We shall use~(V) in the following form: \emph{if $u,w$ are non-adjacent and all their common neighbours have colour $\gamma$, then $u$ and $w$ do not both have colour $\gamma$.}

\section{The graph \texorpdfstring{$G_{11}$}{G11} and the family \texorpdfstring{$G_{11}(t)$}{G11(t)}}\label{sec:g11}

\subsection{The graph \texorpdfstring{$G_{11}$}{G11}}

\begin{definition}\label{def:g11}
$G_{11}$ is the graph with vertex set $\{0,1,\dots,10\}$, whose vertices have the following roles:
\begin{itemize}[itemsep=1pt]
\item the \emph{base vertices} $0,1,2,3$
\item the \emph{caps} $4,5,6,7$
\item the \emph{links} $8,9$
\item the \emph{hub} $10$
\end{itemize}
Its $22$ edges are:
\begin{itemize}[itemsep=1pt]
\item the edges $01,12,23,30$, forming the \emph{base cycle} $0\,1\,2\,3\,0$
\item each cap is adjacent to two consecutive base vertices: $4\sim0,1$,\enspace $5\sim1,2$,\enspace $6\sim2,3$,\enspace $7\sim3,0$
\item each link is adjacent to two caps with no common base neighbour: $8\sim4,6$,\enspace $9\sim5,7$
\item the hub is adjacent to every cap and every link: $10\sim4,5,6,7,8,9$
\end{itemize}
\end{definition}

\begin{figure}[ht]
\centering
\includegraphics[width=0.88\linewidth]{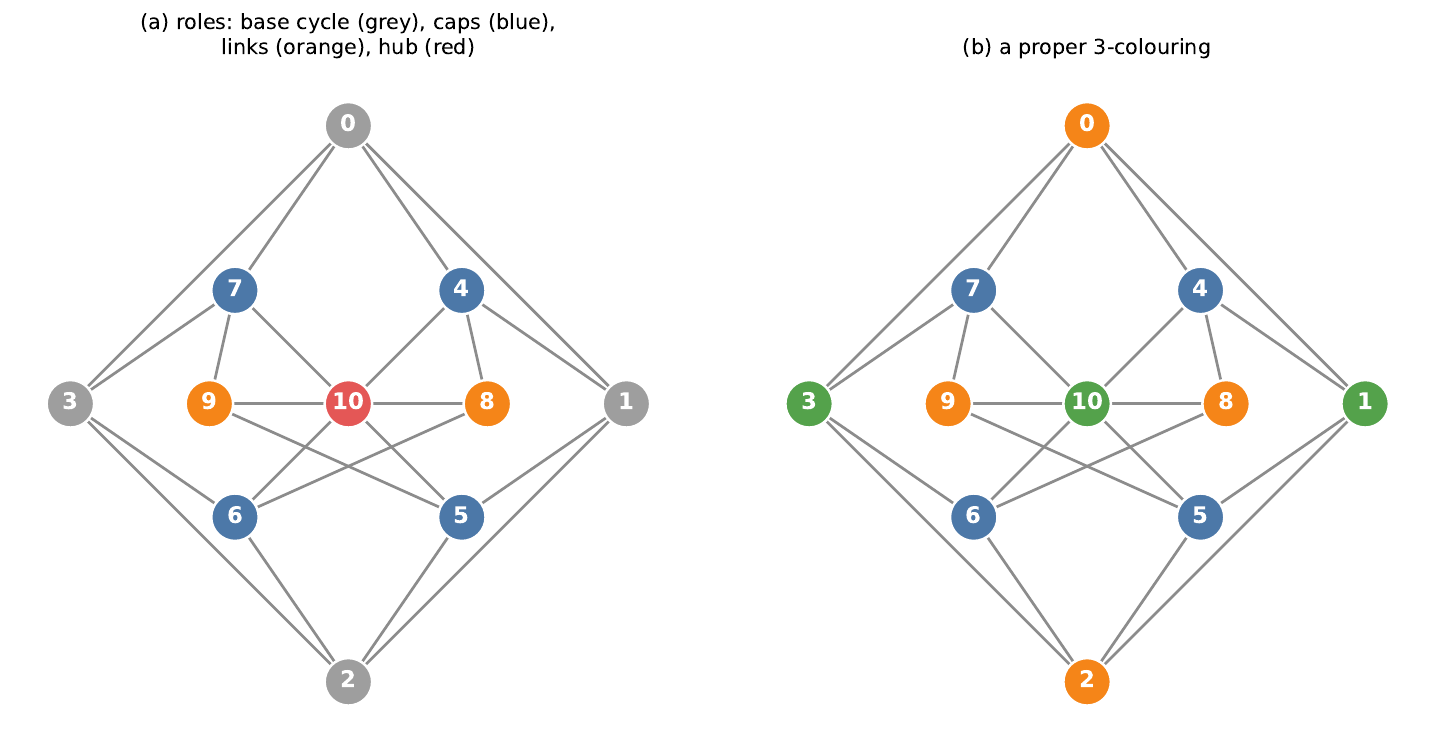}
\caption{The graph $G_{11}$. (a) The roles of the vertices. (b) A proper $3$-colouring with classes $\{4,5,6,7\}$, $\{0,2,8,9\}$, $\{1,3,10\}$.}
\label{fig:g11}
\end{figure}

The graph $G_{11}$ is drawn in Figure~\ref{fig:g11}. We denote the hub $10$ by $h$ and the set $\{0,1,2,3\}$ of base vertices by $B$. For a cap $c$ let $\sigma(c)$ be the set of the two base vertices adjacent to $c$ (they are consecutive on the base cycle), and let $\lambda(c)$ be the link adjacent to $c$:
\[
\begin{array}{c@{\qquad}cccc}
c & 4 & 5 & 6 & 7\\ \hline
\sigma(c) & \{0,1\} & \{1,2\} & \{2,3\} & \{3,0\}\\
\lambda(c) & 8 & 9 & 8 & 9
\end{array}
\]
Thus the two caps of a link have disjoint sets $\sigma$, so that $\sigma(4)\cup\sigma(6)=\sigma(5)\cup\sigma(7)=B$. For two caps $c,c'$ with different links, $\sigma(c)$ and $\sigma(c')$ have exactly one base vertex in common, and we say that $c$ and $c'$ \emph{share} this base vertex. In the cyclic order $4,5,6,7$ consecutive caps share the base vertices $1,2,3,0$, respectively. Table~\ref{tab:nbhd} lists all neighbourhoods.

\begin{table}[ht]
\centering
\begin{tabular}{ll@{\qquad}ll@{\qquad}ll}
\toprule
$v$ & $N(v)$ & $v$ & $N(v)$ & $v$ & $N(v)$\\
\midrule
$0$ & $\{1,3,4,7\}$ & $4$ & $\{0,1,8,10\}$ & $8$ & $\{4,6,10\}$\\
$1$ & $\{0,2,4,5\}$ & $5$ & $\{1,2,9,10\}$ & $9$ & $\{5,7,10\}$\\
$2$ & $\{1,3,5,6\}$ & $6$ & $\{2,3,8,10\}$ & $10$ & $\{4,5,6,7,8,9\}$\\
$3$ & $\{0,2,6,7\}$ & $7$ & $\{0,3,9,10\}$ & &\\
\bottomrule
\end{tabular}
\caption{The neighbourhoods in $G_{11}$. A base vertex $b$ is adjacent to its two neighbours on the base cycle and to the two caps $c$ with $b\in\sigma(c)$; a cap $c$ is adjacent to $\sigma(c)$, $\lambda(c)$ and the hub; a link is adjacent to its two caps and the hub; the hub is adjacent to all caps and both links.}
\label{tab:nbhd}
\end{table}

\noindent\begin{minipage}{\linewidth}
\begin{observation}[Common neighbourhoods]\label{obs:common}
In $G_{11}$, with hub $h$ and base $B$, the following hold.
\begin{enumerate}[label=(\alph*)]
\item Two opposite base vertices (e.g.\ $0$ and $2$) are non-adjacent, and their common neighbours are exactly the other two base vertices ($1$ and $3$).
\item Let $\ell$ be a link, and let $x$ be the other link or a cap not adjacent to $\ell$ (e.g.\ $\ell=8$ and $x=9$ or $x=5$). Then $\ell\not\sim x$, and $h$ is the only common neighbour of $\ell$ and $x$.
\item Let $\ell$ be a link and $b\in B$, and let $c$ be the unique cap with $\lambda(c)=\ell$ and $b\in\sigma(c)$ (e.g.\ $\ell=8$, $b=0$, $c=4$). Then $\ell\not\sim b$, and $c$ is the only common neighbour of $\ell$ and $b$.
\item Let $c,c'$ be caps sharing exactly one base vertex $b$ (e.g.\ $c=6$, $c'=7$, $b=3$). Then $c\not\sim c'$, and the common neighbours of $c$ and $c'$ are exactly $b$ and $h$.
\item In the situation of (d), let $r$ be the base vertex with $\sigma(c')=\{b,r\}$ (in the example, $r=0$). Then $c\not\sim r$, and $b$ is the only common neighbour of $c$ and $r$.
\end{enumerate}
\end{observation}
\end{minipage}

Each item of Observation~\ref{obs:common} can be checked directly from Table~\ref{tab:nbhd}.

\begin{lemma}\label{lem:g11basic}
$G_{11}$ has diameter two and $\chi(G_{11})=3$.
\end{lemma}

\begin{proof}
It is easy to verify from Table~\ref{tab:nbhd} that every two non-adjacent vertices have a common neighbour; as $G_{11}$ is not complete, its diameter is two.

The sets $\{4,5,6,7\}$, $\{0,2,8,9\}$ and $\{1,3,10\}$ are independent (Table~\ref{tab:nbhd}), so $\chi\le3$; and $\{4,8,10\}$ is a triangle, so $\chi\ge3$.
\end{proof}

\begin{figure}[ht]
\centering
\includegraphics[width=0.86\linewidth]{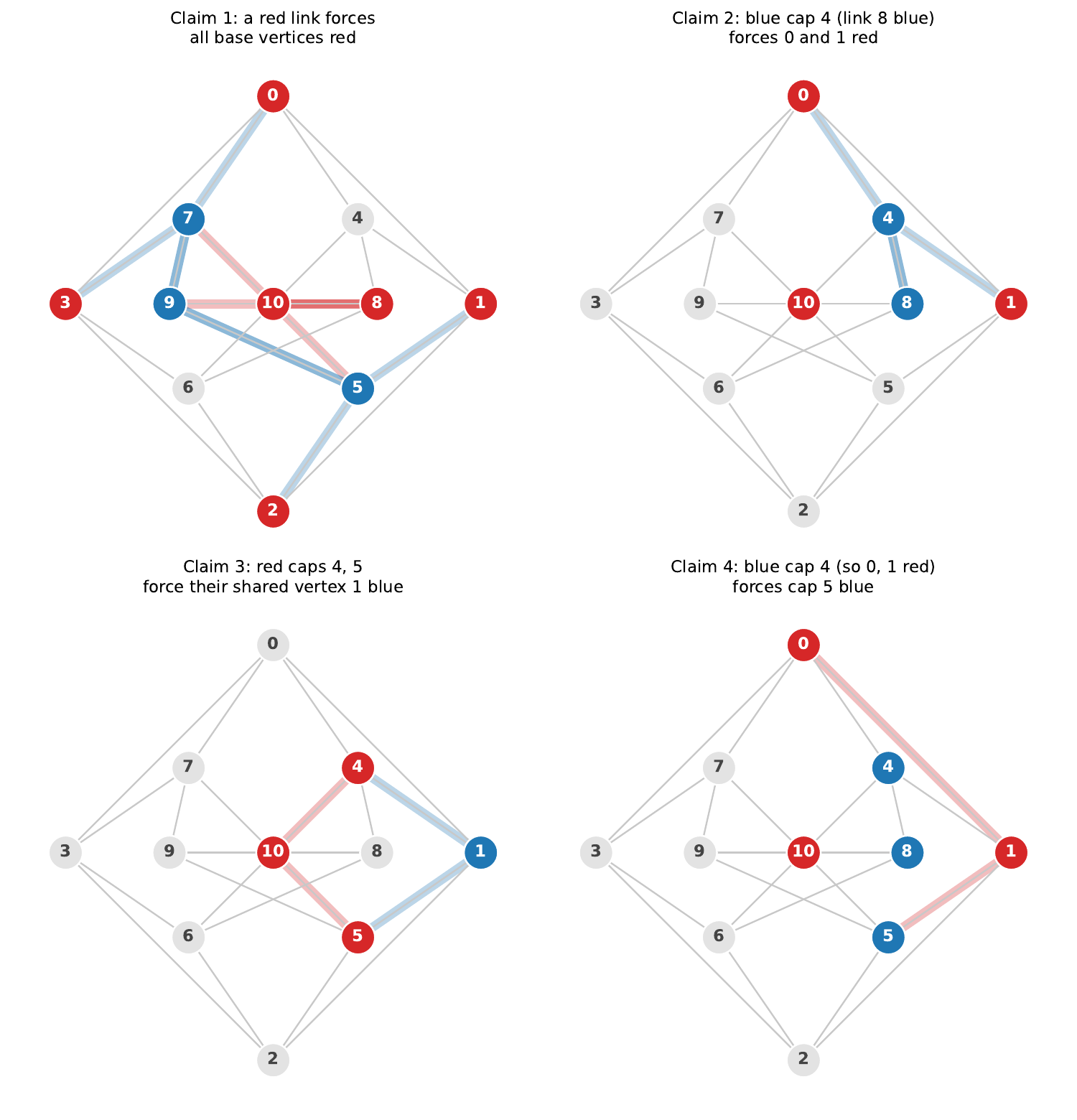}
\caption{The four claims in the proof of Theorem~\ref{thm:g11}, each for one representative case (the hub $10$ is red); each panel title states what is given and what is deduced. Grey vertices are not involved. Each shaded path $u\,x\,w$ marks a pair $u,w$ of non-adjacent vertices together with a common neighbour $x$, shaded in the colour of $x$.}
\label{fig:claims}
\end{figure}

\begin{theorem}\label{thm:g11}
$\cmu(G_{11})=\chi(G_{11})=3$. In particular, $G_{11}$ is a diameter-two graph which is not complete bipartite and satisfies $\cmu(G_{11})=\chi(G_{11})$.
\end{theorem}

\begin{proof}
By Lemma~\ref{lem:g11basic} and Proposition~\ref{prop:upper}, $\cmu\le\chi=3$. Suppose, for a contradiction, that $\cmu(G_{11})\le2$. Fix a colouring $\varphi\colon V(G_{11})\to\{\R,\B\}$ with property~(V). Exchanging the colours if necessary, we may assume that the hub $h$ is red. We prove four claims about how $\varphi$ must colour the links, the caps and the base vertices (illustrated in Figure~\ref{fig:claims}), and then reach a contradiction in two cases. Every contradiction will be with the following fact:
\[
(\star)\qquad \text{the base vertices do not all have the same colour.}
\]
Indeed, the only common neighbours of two opposite base vertices are the other two base vertices (Observation~\ref{obs:common}(a)), so if all four had the same colour, (V) would fail for an opposite pair. Each step below applies~(V) in the same way: for two non-adjacent vertices $u,w$ we observe that $u$ and all common neighbours of $u$ and $w$ have the same colour $\gamma$, and conclude that $w$ does not have colour $\gamma$.

\claimhead{1}{Both links are blue.}
Suppose a link $\ell$ is red, and let $\ell'$ be the other link.
\par\noindent(i) Let $x$ be $\ell'$ or one of the two caps of $\ell'$ (these are not adjacent to $\ell$). The only common neighbour of $\ell$ and $x$ is $h$ (Observation~\ref{obs:common}(b)); as $\ell$ and $h$ are red, $x$ is blue.
\par\noindent(ii) Let $b\in B$. The only common neighbour of $\ell'$ and $b$ is a cap of $\ell'$ (Observation~\ref{obs:common}(c)); as $\ell'$ and that cap are blue by~(i), $b$ is red.
\par\noindent So all base vertices are red, contradicting~$(\star)$.\claimqed

\claimhead{2}{If a cap $c$ is blue, then both base vertices in $\sigma(c)$ are red.}
Let $b\in\sigma(c)$. The only common neighbour of $\lambda(c)$ and $b$ is $c$ (Observation~\ref{obs:common}(c)); as $\lambda(c)$ (blue by Claim~1) and $c$ are blue, $b$ is red.\claimqed

\claimhead{3}{If two red caps $c,c'$ share a base vertex $b$, then $b$ is blue.}
The only common neighbours of $c$ and $c'$ are $b$ and $h$ (Observation~\ref{obs:common}(d)); as $c$, $c'$ and $h$ are red, $b$ is blue.\claimqed

\claimhead{4}{If a cap $c$ is blue and a cap $c'$ shares a base vertex with $c$, then $c'$ is blue.}
Write $\sigma(c)=\{b,s\}$ with $b$ the shared base vertex; by Claim~2, $b$ and $s$ are red. The only common neighbour of $c'$ and $s$ is $b$ (Observation~\ref{obs:common}(e), with the roles of $c$ and $c'$ exchanged); as $s$ and $b$ are red, $c'$ is blue.\claimqed

\medskip\noindent Using the four claims, we now distinguish two cases.

\noindent\emph{Case 1: some cap is blue.} Since consecutive caps in the cyclic order $4,5,6,7$ share a base vertex, repeated application of Claim~4 shows that all four caps are blue. By Claim~2 for the caps $4$ and $6$, all base vertices in $\sigma(4)\cup\sigma(6)=B$ are red, contradicting~$(\star)$.

\noindent\emph{Case 2: all caps are red.} By Claim~3 for the pairs $(4,5)$, $(5,6)$, $(6,7)$, $(7,4)$, which share the base vertices $1,2,3,0$, respectively, all base vertices are blue, contradicting~$(\star)$.

Hence $\cmu(G_{11})\ge3$.
\end{proof}

\subsection{The family \texorpdfstring{$G_{11}(t)$}{G11(t)}}

\begin{definition}\label{def:g11t}
For $t\ge1$, let $G_{11}(t)$ be obtained from $G_{11}$ by replacing the cap $4$ by $t$ pairwise non-adjacent vertices $4^{(1)},4^{(2)},\dots,4^{(t)}$, each adjacent to exactly $N(4)=\{0,1,8,10\}$. Thus $G_{11}(1)$ is $G_{11}$ with the cap $4$ renamed $4^{(1)}$ (see Figure~\ref{fig:family} for $t=3$).
\end{definition}

\begin{figure}[ht]
\centering
\includegraphics[width=0.95\linewidth]{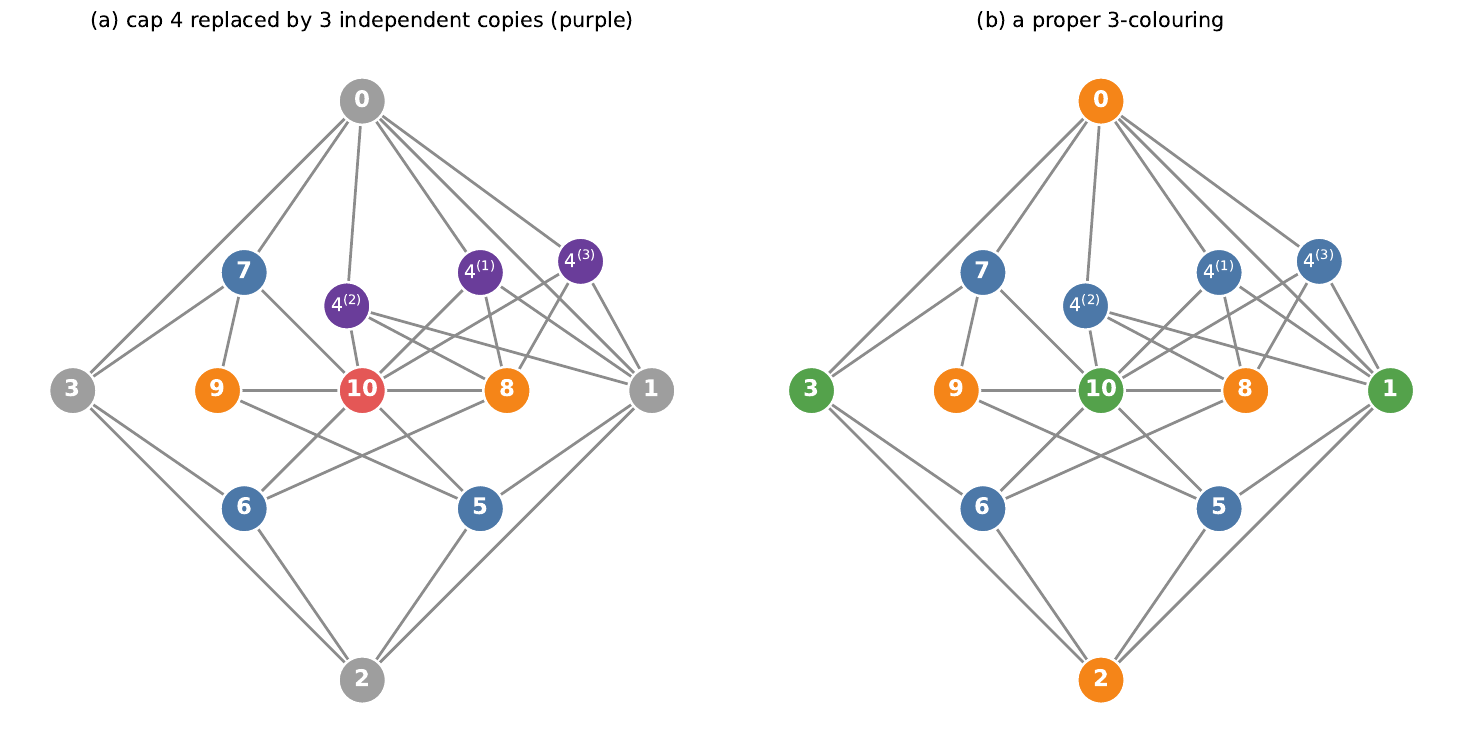}
\caption{The graph $G_{11}(3)$. (a) The cap $4$ is replaced by three pairwise non-adjacent copies $4^{(1)},4^{(2)},4^{(3)}$ with the same neighbourhood $\{0,1,8,10\}$. (b) A proper $3$-colouring with classes $\{4^{(1)},4^{(2)},4^{(3)},5,6,7\}$, $\{0,2,8,9\}$, $\{1,3,10\}$.}
\label{fig:family}
\end{figure}

All copies $4^{(i)}$ are again called caps. The caps of $G_{11}(t)$ fall into four \emph{cap classes}
\[
K_4=\{4^{(1)},\dots,4^{(t)}\},\qquad K_5=\{5\},\qquad K_6=\{6\},\qquad K_7=\{7\},
\]
and we put $\sigma(K_i)=\sigma(i)$ and $\lambda(K_i)=\lambda(i)$ ($i=4,5,6,7$). Let $\pi\colon V(G_{11}(t))\to V(G_{11})$ map every copy $4^{(i)}$ to $4$ and fix all other vertices. By construction,
\begin{equation}\label{eq:proj}
\begin{aligned}
&u\sim w \iff \pi(u)\sim\pi(w) &&\text{whenever } \pi(u)\ne\pi(w),\\
&N(u)=\pi^{-1}\bigl(N(\pi(u))\bigr) &&\text{for every } u,
\end{aligned}
\end{equation}
where $u,w$ are vertices of $G_{11}(t)$, the neighbourhood $N(\pi(u))$ is taken in $G_{11}$, and $\pi^{-1}(S)$ denotes the preimage of a set $S\subseteq V(G_{11})$; two distinct copies $4^{(i)},4^{(j)}$ are non-adjacent. Consequently, for any two vertices $u,w$ of $G_{11}(t)$, the set of common neighbours of $u$ and $w$ is the preimage under $\pi$ of the set of common neighbours of $\pi(u)$ and $\pi(w)$ in $G_{11}$: it is obtained from the latter set by replacing $4$, if present, with the whole class $K_4$. Since the two vertices in each item below have distinct images under $\pi$, this gives the following analogue of Observation~\ref{obs:common} for $G_{11}(t)$.

\begin{observation}[Common neighbourhoods in $G_{11}(t)$]\label{obs:common-t}
In $G_{11}(t)$, with hub $h$ and base $B$, the following hold.
\begin{enumerate}[label=(\alph*$'$)]
\item Two opposite base vertices are non-adjacent, and their common neighbours are exactly the other two base vertices.
\item Let $\ell$ be a link, and let $x$ be the other link or a cap not adjacent to $\ell$. Then $\ell\not\sim x$, and $h$ is the only common neighbour of $\ell$ and $x$.
\item Let $\ell$ be a link and $b\in B$, and let $K$ be the unique cap class with $\lambda(K)=\ell$ and $b\in\sigma(K)$. Then $\ell\not\sim b$, and the common neighbours of $\ell$ and $b$ are exactly the caps in $K$.
\item Let $K,K'$ be distinct cap classes sharing exactly one base vertex $b$, and $x\in K$, $y\in K'$. Then $x\not\sim y$, and the common neighbours of $x$ and $y$ are exactly $b$ and $h$.
\item In the situation of (d$'$), let $r$ be the base vertex with $\sigma(K')=\{b,r\}$. Then $x\not\sim r$, and $b$ is the only common neighbour of $x$ and $r$.
\end{enumerate}
\end{observation}

\begin{lemma}\label{lem:g11t-basic}
For every $t\ge1$, $G_{11}(t)$ has diameter two and $\chi(G_{11}(t))=3$.
\end{lemma}

\begin{proof}
Two non-adjacent vertices with distinct images under $\pi$ are at distance two by Lemma~\ref{lem:g11basic} and~\eqref{eq:proj}, and two copies $4^{(i)},4^{(j)}$ have the common neighbour $h$; as $G_{11}(t)$ is not complete, its diameter is two. Giving every vertex $v$ of $G_{11}(t)$ the colour of $\pi(v)$ in the $3$-colouring of Lemma~\ref{lem:g11basic} yields a proper $3$-colouring (Figure~\ref{fig:family}(b)), since adjacent vertices have distinct, adjacent images by~\eqref{eq:proj}. As $\{4^{(1)},8,10\}$ is a triangle, $\chi(G_{11}(t))=3$.
\end{proof}

\begin{theorem}\label{thm:family}
For every $t\ge1$, the graph $G_{11}(t)$ has $10+t$ vertices, diameter two, and $\cmu(G_{11}(t))=\chi(G_{11}(t))=3$.
\end{theorem}

\begin{proof}
We follow the proof of Theorem~\ref{thm:g11}, with caps replaced by cap classes; Figure~\ref{fig:claims} applies with each cap read as its class. By Lemma~\ref{lem:g11t-basic} and Proposition~\ref{prop:upper}, $\cmu\le\chi=3$. Suppose that $\cmu(G_{11}(t))\le2$, and fix a colouring $\varphi$ with property~(V) in which $h$ is red (exchanging the colours if necessary). As before, by Observation~\ref{obs:common-t}(a$'$),
\[
(\star)\qquad \text{the base vertices do not all have the same colour.}
\]
Call a cap class \emph{blue} if all its caps are blue.

\claimhead{1$'$}{Both links are blue.}
Suppose a link $\ell$ is red, and let $\ell'$ be the other link.
\par\noindent(i) Let $x$ be $\ell'$ or a cap in one of the two classes of $\ell'$. The only common neighbour of $\ell$ and $x$ is $h$ (Observation~\ref{obs:common-t}(b$'$)); as $\ell$ and $h$ are red, $x$ is blue.
\par\noindent(ii) Let $b\in B$. The only common neighbours of $\ell'$ and $b$ are the caps of a class of $\ell'$ (Observation~\ref{obs:common-t}(c$'$)); as $\ell'$ and these caps are blue by~(i), $b$ is red.
\par\noindent So all base vertices are red, contradicting~$(\star)$.\claimqed

\claimhead{2$'$}{If a cap class $K$ is blue, then both base vertices in $\sigma(K)$ are red.}
Let $b\in\sigma(K)$. The only common neighbours of $\lambda(K)$ and $b$ are the caps in $K$ (Observation~\ref{obs:common-t}(c$'$)); as $\lambda(K)$ (blue by Claim~1$'$) and these caps are blue, $b$ is red.\claimqed

\claimhead{3$'$}{If distinct cap classes $K,K'$ share a base vertex $b$, and $x\in K$, $y\in K'$ are red, then $b$ is blue.}
The only common neighbours of $x$ and $y$ are $b$ and $h$ (Observation~\ref{obs:common-t}(d$'$)); as $x$, $y$ and $h$ are red, $b$ is blue.\claimqed

\claimhead{4$'$}{If a cap class $K$ is blue and a cap class $K'\ne K$ shares a base vertex with $K$, then $K'$ is blue.}
Write $\sigma(K)=\{b,s\}$ with $b$ the shared base vertex; by Claim~2$'$, $b$ and $s$ are red. For every $y\in K'$, the only common neighbour of $y$ and $s$ is $b$ (Observation~\ref{obs:common-t}(e$'$), with the roles of $K$ and $K'$ exchanged); as $s$ and $b$ are red, $y$ is blue.\claimqed

\medskip\noindent Using the four claims, we now distinguish two cases.

\noindent\emph{Case 1: some cap class is blue.} Since consecutive classes in the cyclic order $K_4,K_5,K_6,K_7$ share a base vertex, repeated application of Claim~4$'$ shows that all four classes are blue, and by Claim~2$'$ for $K_4$ and $K_6$ all base vertices are red, contradicting~$(\star)$.

\noindent\emph{Case 2: every cap class contains a red cap.} Choose red caps $x_i\in K_i$. By Claim~3$'$ for the consecutive pairs $(x_4,x_5)$, $(x_5,x_6)$, $(x_6,x_7)$, $(x_7,x_4)$, which share the base vertices $1,2,3,0$, respectively, all base vertices are blue, contradicting~$(\star)$.

Hence $\cmu(G_{11}(t))\ge3$.
\end{proof}

\begin{corollary}\label{cor:g11}
The graphs $G_{11}(t)$, $t\ge1$, form an infinite family of diameter-two graphs which are not bipartite and satisfy $\cmu(G_{11}(t))=\chi(G_{11}(t))$.
\end{corollary}

\begin{proof}
This follows from Theorem~\ref{thm:family}, since each $G_{11}(t)$ contains the triangle $\{4^{(1)},8,10\}$.
\end{proof}

\section{Concluding remarks}\label{sec:concl}

Theorems~\ref{thm:g11} and~\ref{thm:family} answer the second question of Problem~\ref{prob:83}: the complete bipartite graphs are not the only diameter-two graphs $G$ with $\cmu(G)=\chi(G)$, and there are infinitely many others. A full characterization appears difficult; already deciding whether $\cmu(G)\le 2$ is NP-complete, even for graphs of diameter four~\cite{BDL2025}. We mention some natural questions.

\begin{enumerate}[itemsep=2pt]
\item Is there, for every $k\ge4$, a diameter-two graph $G$ with $\cmu(G)=\chi(G)=k$?
\item Is there a structural description of the diameter-two graphs with $\cmu(G)=\chi(G)=3$?
\end{enumerate}

\section*{Use of generative AI}
During this work the author used Claude Opus 5.5 (Anthropic), which designed and ran the computer searches that found $G_{11}$, proposed the family $G_{11}(t)$, produced the proofs of Theorems~\ref{thm:g11} and~\ref{thm:family}, wrote the code that produced the figures, and helped prepare the manuscript. The author directed the investigation through questions that shaped its course, verified every step of each proof by hand, and takes full responsibility for the content of the publication.

\end{document}